\documentclass{amsart}

\usepackage[latin1]{inputenc}
\usepackage{tikz}
\usepackage{pgf}
\usepackage{pgfplots}
\usepackage{graphicx}
\usepackage{amsmath}
\usepackage{chicago}
\usepackage{algorithm}
\usepackage{algorithmic}

\newtheorem{theorem}{Theorem}[section]

\newtheorem{lemma}[theorem]{Lemma}

\newcommand{\field}[1]{\ensuremath{\mathbf{#1}}}
\newcommand{\R}{\field{R}}
\newcommand{\C}{\field{C}}
\newcommand{\e}{\mathrm{e}}
\newcommand{\dd}{\ensuremath{\mathrm{d}}}
\newcommand\mnew{m_{\text{new}}}
\newcommand\mave{m_{\text{ave}}}
\newcommand\taunew{\tau_{\text{new}}}
\newcommand\tend{t_{\text{end}}}
\newcommand{\matlab}{\textsc{matlab}}
\newcommand{\mathematica}{\textsc{mathematica}}
\newcommand{\expokit}{\textsc{expokit}}
\newcommand{\expv}{\texttt{expv}}
\newcommand{\expmv}{\texttt{expmv}}
\newcommand{\phiv}{\texttt{phiv}}
\newcommand{\phip}{\texttt{phip}}
\newcommand{\phipm}{\texttt{phipm}}
\newcommand{\expm}{\texttt{expm}}
\newcommand{\odeofs}{\texttt{ode15s}}

\begin{document}

\title{A Krylov subspace algorithm for evaluating the
  $\varphi$-functions appearing in exponential integrators} 
\author{Jitse Niesen}
\address{School of Mathematics \\ 
         University of Leeds \\
         Leeds, LS2 9JT \\ 
         United Kingdom}
\email{jitse@maths.leeds.ac.uk}
\thanks{JN was supported by Australian Research Council grant DP0559083.}
\author{Will M. Wright}
\address{Department of Mathematics and Statistics \\
         La Trobe University \\
         3086 Victoria \\
         Australia}
\email{w.wright@latrobe.edu.au}
\date{11 November 2010}

\maketitle
\markboth{Jitse Niesen and Will M. Wright}{A Krylov subspace algorithm
  for evaluating the $\varphi$-functions}

\begin{abstract}
  We develop an algorithm for computing the solution of a large system
  of linear ordinary differential equations (ODEs) with polynomial
  inhomogeneity. This is equivalent to computing the action of a
  certain matrix function on the vector representing the initial
  condition. The matrix function is a linear combination of the matrix
  exponential and other functions related to the exponential (the
  so-called $\varphi$-functions). Such computations are the major
  computational burden in the implementation of exponential
  integrators, which can solve general ODEs. Our approach is to
  compute the action of the matrix function by constructing a Krylov
  subspace using Arnoldi or Lanczos iteration and projecting the
  function on this subspace. This is combined with time-stepping to
  prevent the Krylov subspace from growing too large. The algorithm is
  fully adaptive: it varies both the size of the time steps and the
  dimension of the Krylov subspace to reach the required accuracy. We
  implement this algorithm in the \matlab\ function \phipm\ and we
  give instructions on how to obtain and use this function. Various
  numerical experiments show that the \phipm\ function is often
  significantly more efficient than the state-of-the-art.
\end{abstract}


\section{Introduction}
\label{sect:intro}

In recent years there has been a resurgence of interest in a class of
numerical methods for the solution of ordinary differential equations
(ODEs) known as exponential integrators. These are intended to be used
on ODEs which can be split into a stiff linear part and a non-stiff
nonlinear part. This splitting can be done once or several times, as
need be. As their name suggests, exponential integrators use the
matrix exponential and various related matrix functions, generally
referred to as $\varphi$-functions, within the numerical integrator.
The computational cost of exponential integrators is dominated by the
need to evaluate these $\varphi$-functions, and this task is the
subject of this paper. For a recent review of exponential integrators,
we refer the reader to \citeN{minchev05aro}.

ODEs of the form exploited by exponential integrators often arise when
semi-discretizing a partial differential equation. Typically, the
matrix appearing in the linear part is large and sparse. For such
matrices, Krylov subspace methods provide an extremely efficient means
of evaluating the action of an arbitrary matrix function on a vector,
without the need to evaluate the matrix function itself. The use of
Krylov subspace approximations for the action of matrix exponential
was pioneered by several authors in the late eighties and early
nineties, notably \citeN{friesner89amf} and \citeN{gallopoulos92eso}.
\citeN{hochbruck97oks} show that for particular classes of matrices,
the convergence of the action of the matrix exponential on a vector is
faster than that for the solution of the corresponding linear system.
This suggests that exponential integrators can be faster than implicit
methods, because computing the action of the matrix exponential and
solving linear systems are the two fundamental operations for
exponential integrators and implicit methods, respectively. This
analysis breathed life back in the class of exponential integrators
invented in the early sixties, which had been abandoned in the eighties
due to their excessive computational expense.

In their well-known paper on computing the matrix exponential,
\citeN{moler03ndw} write: ``The most extensive software for computing
the matrix exponential that we are aware of is \expokit.''  This
refers to the work of \citeN{sidje98esp}, who wrote software for the
computation of the matrix exponential of both small dense and large
sparse matrices. The software uses a Krylov subspace approach in the
large sparse setting. Computing the action of the matrix exponential
is equivalent to solving a linear ODE, and Sidje uses this equivalence
to apply time-stepping ideas from numerical ODE methods in \expokit.
The time-steps are chosen with the help of an error estimate due to
\citeN{saad92aos}.  \citeN{sidje98esp} extends this approach to the
computation of the first $\varphi$-function, which appears in the
solutions of linear ODEs with constant inhomogeneity. This was further
generalized by \citeN{sofroniou07eco} to polynomial inhomogeneities
(or, from another point of view, to general $\varphi$-functions);
their work is included in the latest version of \mathematica.

In all the approaches described above the size of the Krylov subspace
is fixed, generally to thirty. \citeN{hochbruck98eif} explain in their
landmark paper one approach to adapt the size of the Krylov subspace.
In this paper, we develop a solver which combines the time-stepping
ideas of \citeN{sidje98esp} and \citeN{sofroniou07eco} with the
adaptivity of the dimension of the Krylov subspace as described by
\citeN{hochbruck98eif}.  We do not examine the case of small dense
matrices; this has been studied by \citeN{koikari07aea} and
\citeN{skaflestad09tsa}. This algorithm described in this paper can be
used as a kernel for the efficient implementation of certain classes
of exponential integrators. We also recommend the recent book by
\citeN{higham08fom}, which discusses many of the issues related to
matrix functions and their computation.

The approach followed in this paper, reducing large matrices to
smaller ones by projecting them on Krylov subspace, is not the only
game in town. Other possibilities are restricted-denominator rational
Krylov methods \cite{moret07ork}, the real Leja point method
\cite{caliari09ioe}, quadrature formulas based on numerical inversion
of sectorial Laplace transforms \cite{lopez-fernandez10aqb}, and
contour integration \cite{schmelzer07emf}. These methods are outside
the scope of this paper, but we intend to study and compare them in
future work.

The outline of this paper follows. In Section~\ref{sect:phi} we
present several useful results regarding the $\varphi$-functions. The
algorithm we have developed is explained in Section~\ref{sect:krylov},
where we present the Krylov subspace method, show how error
estimation, time-stepping and adaptivity are handled in our algorithm,
and finally give some instructions on how to use our implementation.
Several numerical experiments are given in Section~\ref{sect:numer}
followed by some concluding remarks and pointers towards future work
in Section~\ref{sect:concl}.


\section{The $\varphi$-functions}
\label{sect:phi}

Central to the implementation of exponential integrators is the
efficient and accurate evaluation of the matrix exponential and other
$\varphi$-functions. These $\varphi$-functions are defined for scalar
arguments by the integral representation
\begin{equation}
  \label{eq:phifunctions}
  \varphi_0(z) = \e^z, \quad \varphi_{\ell}(z) = \frac{1}{({\ell}-1)!}
  \int_0^1 \e^{(1- \theta)z} \theta^{\ell-1}\, \dd\theta, \qquad \ell
  = 1, 2, \dots, \, z \in \C.
\end{equation}
For small values of $\ell$, these functions are
\begin{equation*}
  \varphi_1(z) = \frac{\e^z - 1}{z}, \quad
  \varphi_2(z) = \frac{\e^z - 1 - z}{z^2}, \quad
  \varphi_3(z) = \frac{\e^z - 1 - z - \frac{1}{2}z^2}{z^3}.
\end{equation*}
The $\varphi$-functions satisfy the recurrence relation
\begin{equation}
  \label{eq:phirelation}
  \varphi_\ell(z) = z \varphi_{\ell + 1}(z) + \frac{1}{\ell!}, \qquad
  \ell = 1, 2, \dots. 
\end{equation}
The definition can then be extended to matrices instead of scalars
using any of the available definitions of matrix functions, such as
that based on the Jordan canonical form \cite{horn91tim,higham08fom}.

Every stage in an exponential integrator can be expressed as a linear
combination of $\varphi$-functions acting on certain vectors:
\begin{equation}
  \label{eq:lincomb}
  \varphi_0(A) b_0 + \varphi_1(A) b_1 + \varphi_2(A) b_2 + \cdots +
  \varphi_p(A) b_p.
\end{equation}
Here $p$ is related to the order of the exponential integrator,
typically taking values less than five. $A$ is a matrix, often the
Jacobian for exponential Rosenbrock-type methods or an approximation
to it for methods based on the classical linear/non-linear splitting;
usually, $A$ is large and sparse.

We need to compute expressions of the form~\eqref{eq:lincomb} several
times in each step that the integrator takes, so there is a need to
evaluate these expressions efficiently and accurately. This is the
problem taken up in this paper. We would like to stress that any
procedure for evaluating~\eqref{eq:lincomb}, such as the one described
here, is independent of the specific exponential integrator
used and can thus be re-used in different integrators. The exponential
integrators differ in the vectors $b_0,\ldots,b_p$ appearing
in~\eqref{eq:lincomb}.

The following lemma 
gives a formula for the exact solution of linear differential
equations with polynomial inhomogeneity. This result partly explains
the important role that $\varphi$-functions play in exponential
integrators (see~\citeN{minchev05aro} for more details).  The lemma
also provides the background for the time-stepping procedure for the
evaluation of \eqref{eq:lincomb} which we develop in
\S\ref{sect:timestep}.

\begin{lemma}[(\citeN{skaflestad09tsa})]
  \label{lem:exactsolution}
  The solution of the non-auto\-nomous linear initial value problem 
  \begin{equation}
    \label{eq:nonautonomous}
    u'(t) = A u(t) + \sum_{j=0}^{p-1} \frac{t^{j}}{j!} b_{j+1}, 
    \qquad u(t_k) = u_k,
  \end{equation}
  is given by
  \begin{equation*}
    \label{eq:exactsolution}
    u(t_k+\tau_k) = \varphi_0(\tau_k A) u_k + 
    \sum_{j=0}^{p-1} \sum_{\ell=0}^{j}  \frac{t_k^{j-\ell}}{(j-\ell)!}
    \tau_k^{\ell+1} \varphi_{\ell+1}(\tau_kA) b_{j+1},
  \end{equation*}
  where the functions $\varphi_\ell$ are defined
  in~\eqref{eq:phifunctions}.
\end{lemma}

\begin{proof}
  Recall that $\varphi_0$ denotes the matrix exponential.  Using
  $\varphi_0((t_k-t)A)$ as an integrating factor for
  \eqref{eq:nonautonomous} we arrive at
  \begin{equation*}
    \begin{split}
      u(t_k + \tau_k) 
      & = \varphi_0(\tau_k A) u_k + \varphi_0(\tau_k A) \int_0^{\tau_k}
      \varphi_0(-s A) \sum_{j=0}^{p-1} \frac{(t_k+\tau_k)^{j}}{j!}
      b_{j+1} \, \dd s \\
      & = \varphi_0(\tau_k A) u_k + \varphi_0(\tau_k A)
      \int_0^{\tau_k} \varphi_0(-s A) \sum_{j=0}^{p-1} \sum_{\ell=0}^{j}
      \frac{t_k^{j-\ell} s^\ell}{\ell!(j-\ell)!} b_{j+1} \, \dd s. \\
      \intertext{Now change the integration variable, $s=\theta \tau_k$,
        and apply the definition~\eqref{eq:phifunctions} of the
        $\varphi$-functions.} 
      & = \varphi_0(\tau_k A) u_k + \sum_{j=0}^{d-1} \sum_{\ell=0}^{j}
      \frac{t_k^{j-\ell}}{(j-\ell)!} \tau_k^{\ell+1}
      \left(\frac{1}{\ell!}\int_0^1 \varphi_0((1-\theta) \tau_k A)
        \theta^{\ell}\, \dd \theta \right) b_{j+1} \\
      & = \varphi_0(\tau_k A) u_k + \sum_{j=0}^{p-1} \sum_{\ell=0}^{j}
      \frac{t_k^{j-\ell}}{(j-\ell)!} \tau_k^{\ell+1}
      \varphi_{\ell+1}(\tau_k A) b_{j+1}. 
    \end{split}
  \end{equation*} 
\end{proof}


\section{The algorithm}
\label{sect:krylov}

This section describes the details of our algorithm for evaluating
expressions of the form~\eqref{eq:lincomb} as implemented in the
\matlab\ function \phipm. In the first part of this section we explain
how Krylov subspace techniques can be used to reduce large matrices to
small ones when evaluating matrix functions. An estimate of the error
committed in the Krylov subspace approximation is essential for an
adaptive solver; this is dealt with in the second part. Then we
discuss how to split up the computation of the $\varphi$-functions
into several steps. Part four concerns the possibility of adapting the
Krylov subspace dimension and the size of the steps using the error
estimate from the second part, and the interaction between both forms
of adaptivity. Finally, we explain how to use the implementation
provided in the \phipm\ function.

\subsection{The basic method}
\label{sect:basic}

We start by considering how to compute $\varphi_p(A)v$, where $A$ is
an $n \times n$ matrix (with $n$ large) and $v \in \R^n$. This is one
of the terms in~\eqref{eq:lincomb}. We will use a Krylov subspace
approach for this task.

The idea behind this Krylov subspace approach is quite simple. The
vector $\varphi_p(A)v$ lives in $\R^n$, which is a big space. We
approximate it in a smaller space of dimension~$m$. This smaller space
is the Krylov subspace, which is given by
\begin{equation*}
  K_m = \operatorname{span} \{v, Av, A^2v, \ldots, A^{m-1}v \}.
\end{equation*}
However, the vectors~$A^jv$ form a bad basis for the Krylov subspace
because they point in almost the same direction as the dominant
eigenvector of~$A$; thus, the basis vectors are almost linearly
dependent. In fact, computing these successive products is the power
iteration method for evaluating the dominant eigenvector. Therefore,
we apply the (stabilized) Gram--Schmidt procedure to get an
orthonormal basis of the Krylov subspace:
\begin{equation*}
  K_m = \operatorname{span} \{ v_1, v_2, \ldots, v_m \}.
\end{equation*}
Let $V_m$ denote the $n$-by-$m$ matrix whose columns are
$v_1,\dots,v_m$. Then the $m$-by-$m$ matrix $H_m = V_m^T A V_m$ is the
projection of the action of~$A$ to the Krylov subspace, expressed in
the basis $\{v_1,\dots,v_m\}$. The \emph{Arnoldi iteration}
(Algorithm~\ref{alg:arnoldi}) computes the matrices~$H_m$
and~$V_m$, see~\citeN{saad92aos}.

\begin{algorithm}
  \caption{The Arnoldi iteration.}
  \label{alg:arnoldi}
  \begin{algorithmic}[0]
    \STATE $v_1 = v / \|v\|$
    \FOR{$j = 1, \dots, m$}
    \STATE $w = Av_j$
    \FOR{$i = 1, \dots, j$}
    \STATE $h_{i,j} = v_i^T w; \, w = w - h_{i,j} v_i$ 
    \ENDFOR
    \STATE $h_{j+1,j} = \|w\|; \,  v_{j+1} = w / h_{j+1,j}$
    \ENDFOR
  \end{algorithmic}
\end{algorithm}

It costs $\tfrac32(m^2-m+1)n$ floating point operations (flops) and
$m$~products of $A$ with a vector to compute the matrices~$H_m$
and~$V_m$. The cost of one of these matrix-vector products depends on
the sparsity of~$A$; the straightforward approach uses $2N_A$ flops
where $N_A$ is the number of nonzero entries in the matrix~$A$.

The projection of the action of $A$ on the Krylov subspace~$K_m$ in
the standard basis of~$\R^n$ is $V_m H_m V_m^T$. We now approximate
$\varphi_p(A)v$ by $\varphi_p(V_m H_m V_m^T) v$. Since $V_m^TV_m=I_m$
and $V_mV_m^Tv=v$ we have $\varphi_p(V_m H_m V_m^T) v = V_m
\varphi_p(H_m) V_m^T v$.  Finally, $V_m^T v = \|v\| e_1$, where $e_1$
is the first vector in the standard basis.  Taking everything
together, we arrive at the approximation
\begin{equation}
  \label{eq:phipapprox}
  \varphi_p(A) v \approx \beta V_m \varphi_p(H_m) e_1, 
  \quad \beta = \|v\|. 
\end{equation}
The advantage of this formulation is that the matrix~$H_m$ has size
$m$-by-$m$ and thus it is much cheaper to evaluate $\varphi_p(H_m)$
than~$\varphi_p(A)$.

The matrix~$H_m$ is Hessenberg, meaning that the $(i,j)$~entry
vanishes whenever $i > j+1$. It is related to the matrix~$A$ by the
relation $H_m = V_m^T A V_m$. If $A$ is symmetric, then $H_m$ is both
symmetric and Hessenberg, which means that it is tridiagonal. In that
case, we denote the matrix by~$T_m$, and we only need traverse the
$i$-loop in the Arnoldi iteration twice: once for $i=j-1$ and once for
$i=j$. The resulting algorithm is known as the \emph{Lanczos
  iteration} (Algorithm~\ref{alg:lanczos});
see~\citeN[Alg~36.1]{trefethen97nla}. It takes $3(2m-1)n$ flops and
$m$~products of $A$ with a vector to compute~$T_m$ and~$V_m$ when $A$
is symmetric.

\begin{algorithm}
  \caption{The Lanczos iteration.}
  \label{alg:lanczos}
  \begin{algorithmic}[0]
    \STATE $t_{1,0} = 0; \, v_0 = 0; \, v_1 = v / \|v\|$
    \FOR{$j = 1, \dots, m$}
    \STATE $w = Av_j$; \, $t_{j,j} = v_j^Tw$ 
    \STATE $w = w - t_{j,j-1} v_{j-1} - t_{j,j} v_j$
    \STATE $t_{j,j-1} = \|w\|$; \, $t_{j-1,j} = \|w\|$
    \STATE $v_{j+1} = w/t_{j,j-1}$
    \ENDFOR
  \end{algorithmic}
\end{algorithm}

The Krylov subspace algorithm reduces the problem of computing
$\varphi_p(A)v$ where $A$ is a big $n$-by-$n$ matrix to that of
computing $\varphi_p(H_m)e_1$ where $H_m$ is a smaller $m$-by-$m$
matrix. \citeN{skaflestad09tsa} describe a modified
scaling-and-squaring method for the computation of $\varphi_p(H_m)$.
However, this method has the disadvantage that one generally also has
to compute $\varphi_0(H_m), \dots, \varphi_{p-1}(H_m)$. It is usually
cheaper to compute the matrix exponential $\exp(\hat{H}_m)$ of a
slightly larger matrix~$\hat{H}_m$, following an idea
of~\citeN[Prop.~2.1]{saad92aos}, generalized
by~\citeN[Thm.~1]{sidje98esp} to $p>1$. Indeed, if we define the
augmented matrix~$\hat{H}_m$ by
\begin{equation}
  \label{eq:hathm}
  \hat{H}_m = \begin{bmatrix}
    H_m & e_1 & 0 \makebox[0pt][l]{\hspace{3ex}$m$\text{ rows}} \\
    0 & 0 & I \makebox[0pt][l]{\hspace{3ex}$p-1$\text{ row}s} \\
    0 & 0 & 0 \makebox[0pt][l]{\hspace{3ex}$1$\text{ row}} \\
  \end{bmatrix}
\end{equation}
then the top $m$ entries of the last column of~$\exp(\hat{H}_m)$ yield
the vector $\varphi_p(H_m)e_1$. Finally, we compute the matrix
exponential~$\exp(\hat{H}_m)$ using the degree-13 diagonal Pad\'e
approximant combined with scaling and squaring as advocated by
\citeN{higham05tsa}; this is the method implemented in the function
\expm\ in \matlab\ Version~7.2 (R2006a) and later. In contrast,
\expokit\ uses the degree-14 uniform rational Chebyshev approximant
for symmetric negative-definite matrices and the degree-6 diagonal
Pad\'e approximant for general matrices, combined with scaling and
squaring. We choose not to deal with the negative-definite matrices
separately, but consider this as a possible extension to consider at a
later date. \citeN{koikari09obs} recently developed a new variant of
the Schur--Parlett algorithm using three-by-three blocking.

The computation of~$\exp(\hat{H}_m)$ using Higham's method requires
one matrix division (costing $\frac83(m+p)^3$ flops) and $6 + \lceil
\log_2(\|\hat{H}_m\|_1/5.37) \rceil_+$ matrix multiplications (costing
$2(m+p)^3$ flops each), where $\lceil x \rceil_+$ denotes the smallest
nonnegative integer larger than~$x$. Thus, the total cost of computing
the matrix exponential of~$\hat{H}_m$ is $M(\hat{H}_m)\,(m+p)^3$ where
\begin{equation}
  \label{eq:8}
  M(A) = \frac{44}3 + 2 \left\lceil \log_2 \frac{\|A\|_1}{5.37}
  \right\rceil_+. 
\end{equation}
Only the last column of the matrix exponential~$\exp(\hat{H}_m)$ is
needed; it is natural to ask if this can be exploited. Alternatively
one could ask whether the computation of~$\exp(\hat{H}_m)$ using the
scaling-and-squaring algorithm can be modified to take advantage of
the fact that $\hat{H}_m$ is Hessenberg
(\citeN[Prob.~13.6]{higham08fom} suggests this as a research problem).
The algorithm suggested by~\citeN[Alg.~2.3]{higham05tsa} computes the
matrix exponential to machine precision (Higham considers IEEE single,
double and quadruple precision). The choice of the degree of the
Pad\'e approximation and the number of scaling steps is intimately
connected with this choice of precision. However, we generally do not
require this accuracy. It would be interesting to see if significant
computational savings can be gained in computing the matrix
exponential by developing an algorithm with several other choices of
precision.

\subsection{Error estimation}
\label{sect:error}

\citeN[Thm.~5.1]{saad92aos} derives a formula for the error in the
Krylov subspace approximation~\eqref{eq:phipapprox} to~$\varphi_p(A)
v$ in the case $p=0$. This result was generalized
by~\citeN[Thm.~2]{sidje98esp} to $p>0$. Their result states that
\begin{equation}
  \label{eq:errorseries}
  \varphi_p(A) v - \beta V_m \varphi_p(H_m) e_1
  = \beta \sum_{j=p+1}^\infty h_{m+1,m} e_m^T \varphi_j(H_m) e_1
  A^{j-p-1} v_{m+1}.
\end{equation}
The first term in the series on the right-hand side does not involve
any multiplications with the matrix~$A$ and the Arnoldi iteration
already computes vector~$v_{m+1}$, so this term can be computed
without too much effort. We use this term as an error estimate for the
approximation~\eqref{eq:phipapprox}:
\begin{equation}
  \label{eq:errest}
  \varepsilon 
  = \| \beta h_{m+1,m} e_m^T \varphi_{p+1}(H_m) e_1 v_{m+1} \|
  = \beta |h_{m+1,m}| \, [\varphi_{p+1}(H_m)]_{m,1}.
\end{equation}
Further justification for this error estimate is given
by~\citeN{hochbruck98eif}.

We also use this error estimate as a corrector: instead
of~\eqref{eq:phipapprox}, we use
\begin{equation}
  \label{eq:phipapprox2}
  \varphi_p(A) v \approx \beta V_m \varphi_p(H_m) e_1
  + \beta h_{m+1,m} e_m^T \varphi_{p+1}(H_m) e_1 v_{m+1}.
\end{equation}
If $\varphi_{p+1}(H_m) e_1$ is computed using the augmented
matrix~\eqref{eq:hathm}, then $\varphi_p(H_m) e_1$ also appears in the
result, so we only need to exponentiate a matrix of size $m+p+1$.  The
approximation~\eqref{eq:phipapprox2} is more accurate,
but~$\varepsilon$ is no longer a real error estimate. This is
acceptable because, as explained below, it is not used as an error
estimate but only for the purpose of adaptivity.

\citeN{sidje98esp} proposes a more accurate error estimate which also
uses the second term of the series in~\eqref{eq:errorseries}.
However, the computation of this requires a matrix-vector product and
the additional accuracy is in our experience limited. We thus do not
use Sidje's error estimate.

\subsection{Time-stepping}
\label{sect:timestep}

\citeN[Thm.~4.7]{saad92aos} proves that the error of the
approximation~\eqref{eq:phipapprox} satisfies the bound
\begin{equation}
  \label{eq:1}
  \|\text{Error}\| \le \frac{2 \beta (\rho(A))^m}{m!} (1+o(1))
\end{equation}
for sufficiently large~$m$, where $\rho(A)$ denotes the spectral
radius of~$A$.  This bound deteriorates as $\rho(A)$ increases,
showing that the dimension~$m$ of the Krylov subspace has to be large
if $\rho(A)$ is large. \citeN{sidje98esp} proposes an alternative
method for computing $\varphi_1(A)v$ when $\rho(A)$ is large, based on
time-stepping. This approach is generalized by~\citeN{sofroniou07eco}
for general $\varphi$-functions.

The main idea behind this time-stepping procedure is that
$\varphi_p(A)v$ solves a non-autonomous linear ODE. More generally,
Lemma~\ref{lem:exactsolution} states that the function
\begin{equation}
  \label{eq:lincomb2}
  u(t) = \varphi_0(tA) b_0 + t \varphi_1(tA) b_1 + t^2 \varphi_2(tA)
  b_2 + \cdots + t^p \varphi_p(tA) b_p,
\end{equation} 
is the solution of the differential equation
\begin{equation}
  \label{eq:phide}
  u'(t) = A u(t) + b_1 + tb_2 + \cdots + \frac{t^{p-1}}{(p-1)!} b_p, \qquad
  u(0) = b_0. 
\end{equation}
We now use a time-stepping method to calculate~$u(\tend)$ for some
$\tend\in\R$. If we want to compute an expression of the
form~\eqref{eq:lincomb}, we set $\tend=1$.  Split the time
interval~$[0,\tend]$ by introducing a grid $0 = t_0 < t_1 < \ldots <
t_n = \tend$. To advance the solution, say from~$t_k$ to~$t_{k+1}$, we
need to solve the differential equation~\eqref{eq:phide} with the
value of~$u(t_k)$ as initial condition. The relation between~$u(t_k)$
and $u(t_{k+1})$ is given in Lemma~\ref{lem:exactsolution}.
Rearranging this expression gives
\begin{equation}
  \label{eq:steppre}
  u(t_{k+1}) = \varphi_0(\tau_k A) u(t_k) + \sum_{i=1}^p \tau_k^i
  \varphi_i(\tau_k A) \sum_{j=0}^{p-i} \frac{t_k^j}{j!} b_{i+j},
\end{equation}
where $\tau_k = t_{k+1}-t_k$. However, we do not need to evaluate all
the $\varphi$-functions. The recurrence relation $\varphi_{q}(A) =
\varphi_{q+1}(A)A + \frac{1}{q!}I$ implies that
\begin{equation*}
  \varphi_q(A) = \varphi_p(A) A^{p-q} + \sum_{j=0}^{p-q-1}
  \frac{1}{(q+j)!} A^j, \qquad q=0,1,\ldots,p-1.
\end{equation*}
Substituting this in~\eqref{eq:steppre} yields
\begin{equation}
  \label{eq:step}
  u(t_{k+1}) = \tau_k^p \varphi_p(\tau_k A) w_p 
  + \sum_{j=0}^{p-1} \frac{\tau_k^j}{j!} w_j,
\end{equation}
where the vectors~$w_j$ are given by 
\begin{equation*}
  w_j = A^j u(t_k) + \sum_{i=1}^j A^{j-i} \sum_{\ell=0}^{j-i}
  \frac{t_k^\ell}{\ell!} b_{i+\ell}, \qquad j=0,1,\ldots,p.
\end{equation*}
This is the time-stepping method for computing~\eqref{eq:lincomb}. 

The computational cost of this method is as follows. At every step, we
need to compute the vectors~$w_j$ for $j=0,\ldots,p$, the action
of~$\varphi_p(\tau_k A)$ on a vector, $p+1$~scalar multiplications of
a vector of length~$n$ and $p$~vector additions. The vectors~$w_j$
satisfy the recurrence relation
\begin{equation}
  \label{eq:wrecurrence}
  w_0 = u(t_k) \quad\text{and}\quad
  w_j = Aw_{j-1} + \sum_{\ell=0}^{p-j} \frac{t_k^\ell}{\ell!}
  b_{j+\ell}, \qquad j=1,\ldots,p,
\end{equation}
and hence their computation requires $p$ multiplications of $A$ with a
vector, $p$~scalar multiplications, and $p$~vector additions. 

One reason for developing this time-stepping method is to reduce the
dimension of the Krylov subspace. If the spectrum of $A$ is very
large, multiple time-steps may be required. We intend, in the future,
to compare this approach with the approach described by
\citeN{skaflestad09tsa} which evaluates the matrices $\varphi_0(H_m),
\ldots, \varphi_{p+1}(H_m)$ directly. The latter method may have
computational advantages, particularly if $b_0,\ldots,b_p$ are equal
or zero.  For values of $p$ that we are interested in, that is less
than five, we have not noticed any loss of accuracy using this
approach. We intend to look into this issue more thoroughly in future
investigations, especially in light of the paper \cite{almohy10cta},
which appeared during the review process, which noticed that for large
values of $p$ accuracy can be lost.

We choose an initial step size similar to the one suggested in
\expokit, except we increase the rather conservative estimate by an
order of magnitude, to give
\begin{equation}
  \label{eq:tau_0}
  \tau_0 = \frac{10}{\|A\|_\infty} \left(\frac{ \mathrm{Tol} \, \bigl( (\mave
          + 1) / \e \bigr)^{\mave + 1} \sqrt{2\pi(\mave + 1)}}{4
      \|A\|_\infty \|b_0\|_\infty} \right)^{1 / \mave},  
\end{equation}
where Tol is the user defined tolerance and $\mave$ is the
average of the input and maximum allowed size of the Krylov subspace.

\subsection{Adaptivity}
\label{sect:adaptive}

The procedure described above has two key parameters, the
dimension~$m$ of the Krylov subspace and the time-step~$\tau$. These
need to be chosen appropriately. As we cannot expect the user to make
this choice, and the optimal values may change, the algorithm needs to
determine~$m$ and~$\tau$ adaptively.

We are using a time-stepping method, so adapting the step size~$\tau$
is similar to adaptivity in ODE solvers. This has been studied
extensively. It is described by \citeN[\S 39]{butcher08nmf} and
\citeN[\S II.4]{hairer93sod}, among others. The basic idea is as
follows. We assume that the time-stepping method has order~$q$, so
that the error is approximately~$C\tau^{q+1}$ for some
constant~$C$. We somehow compute an error estimate~$\varepsilon$ and
choose a tolerance~Tol that the algorithm should satisfy. Then the
optimal choice for the new step size is
\begin{equation}
  \label{eq:taunew}
  \taunew = \tau_k \left( \frac{1}{\omega} \right)^{1/(q+1)} 
  \quad\text{where}\quad 
  \omega = \frac{\tend \|\varepsilon\|}{\tau_k \cdot \mathrm{Tol}}.
\end{equation}
The factor~$\tend$ is included so that the method is invariant under
time scalings. Usually, a safety factor~$\gamma$ is added to ensure
that the error will probably satisfy the error tolerance, changing the
formula to $\taunew = \tau_k (\gamma/\omega)^{1/(q+1)}$.  Common
choices are $\gamma = 0.25$ and $\gamma = 0.38$. However, in our case
the consequence of rejecting a step is that we computed the matrix
exponential in vain, while in ODE solvers the whole computation has to
be repeated when a step is rejected. We may thus be more adventurous
and therefore we take $\gamma = 0.8$.

In our scheme, the error estimate~$\varepsilon$ is given
by~\eqref{eq:errest}. However, what is the order~$q$ for our scheme?
The \textsl{a~priori} estimate~\eqref{eq:1} suggests that the order
equals the dimension~$m$ of the Krylov subspace. Experiments confirm
that the error is indeed proportional to~$\tau^{m+1}$ in the limit
$\tau\to0$. However, for finite step size the error is better
described by $\tau^q$ with a smaller exponent~$q$. Let us call the
exponent~$q$ which provides the best fit around a given value step
size~$\tau$ the ``heuristic order'', for lack of a better term.

We can estimate the heuristic order if we have attempted two step
sizes during the same step, which happens if we have just rejected a
step and reduced the step size. The estimate for the heuristic order
is then
\begin{equation}
  \label{eq:6}
  \hat{q} = \frac{\log(\tau / \tau_{\text{old}})}%
  {\log(\|\varepsilon\| / \|\varepsilon_{\text{old}}\|)} - 1,
\end{equation}
where $\varepsilon$ and~$\varepsilon_{\text{old}}$ denote the error
estimates produced when attempting step size~$\tau$
and~$\tau_{\text{old}}$, respectively. In all other cases, we use
$\hat{q} = \frac14m$; there is no rigorous argument behind the choice
of~$\frac14$ but it seems to yield good performance in practice.

With this estimate for the heuristic order, we compute the suggested
new step size as
\begin{equation}
  \label{eq:2}
  \taunew = \tau_k \left( \frac{\gamma}{\omega} \right)^{1/(\hat{q}+1)}.
\end{equation}
The other parameter that we want to adapt is the Krylov subspace
dimension~$m$. The error bound~\eqref{eq:1} suggests that, at least
for modest changes of~$m$, the error is approximately equal to
$C\kappa^{-m}$ for some values of~$C$ and~$k$. Again, we can estimate
$\kappa$ if we have error estimates corresponding to two different
values of~$m$:
\begin{equation}
  \label{eq:7}
  \hat\kappa = \left( \frac{\|\varepsilon\|}{\|\varepsilon_\text{old}\|}
  \right)^{1/(m_\text{old}-m)}.  
\end{equation}
If this formula cannot be used, then we take $\hat\kappa = 2$. Given
this estimate, the minimal~$m$ which satisfies the required tolerance
is given by
\begin{equation}
  \label{eq:3}
  \mnew = m + \frac{\log(\omega/\gamma)}{\log\hat\kappa}.
\end{equation}
We now have to choose between two possibilities: either we keep~$m$
constant and change~$\tau$ to~$\taunew$, or we keep~$\tau$ constant
and change~$m$ to~$\mnew$. We will pick the cheapest option. To
advance from~$t_k$ to~$t_{k+1}$, we need to evaluate~\eqref{eq:step}.
Computation of the vectors~$w_j$ requires $2(p-1)(N_A+n)$ flops. Then,
we need to do $m$~steps of the Arnoldi algorithm, for a cost of
$\tfrac32(m^2-m+1)n + 2mN_A$ flops. If $A$ is symmetric, we will use
the Lanczos algorithm and the costs drops to $3(2m-1)n + 2mN_A$ flops.
To compute $\varphi_p(\tau_kA)w_p$ in~\eqref{eq:step}
using~\eqref{eq:phipapprox2}, we need to exponentiate a matrix of size
$m+p+1$, costing $M(\hat{H}_m)\,(m+p+1)^3$ flops with $M(\hat{H}_m)$
given by~\eqref{eq:8}. Finally, the scalar multiplications and
vector additions in~\eqref{eq:step} requires a further $(2p+1)n$~flops.
All together, we find that the cost of a single step is
\begin{equation}
  \label{eq:4}
  C_1(m) = \begin{cases}
    (m+p)N_A + 3(m+p)n + M(H_m) \, (m+p+1)^3,
    & \text {for Lanczos;} \\[1ex]
    (m+p)N_A + (m^2+3p+2)n + M(H_m) \, (m+p+1)^3,
    & \text {for Arnoldi.}
  \end{cases}
\end{equation}
This needs to be multiplied with the number of steps required to go
from the current time~$t_k$ to the end point $t=\tend$. So the total
cost is
\begin{equation}
  \label{eq:5}
  C(\tau,m) = \left\lceil \frac{\tend-t_k}{\tau} \right\rceil C_1(m).
\end{equation}
We compute $C(\taunew,m)$ and $C(\tau,\mnew)$ according to this
formula. If $C(\taunew,m)$ is smaller, then we change the time-step
to~$\taunew$ and leave $m$ unchanged. However, to prevent too large
changes in~$\tau$ we restrict it to change by no more than a
factor~5. Similarly, if $C(\taunew,m)$ is smaller, then $\tau$ remains
as it is and we change $m$ to~$\mnew$, except that we restrict it to
change by no more as a factor~$\frac43$ (this factor is chosen
following~\citeN{hochbruck98eif}).

Finally, the step is accepted if $\omega > \delta$ where $\delta =
1.2$. Thus, we allow that the tolerance is slightly exceeded. The idea
is that our adaptivity procedure aims to get $\omega$ down to~$\gamma
= 0.8$, so that usually we stay well below the tolerance and hence we
may permit ourselves to exceed it occasionally.  The resulting
algorithm is summarized in Algorithms~\ref{alg:mainalg}
and~\ref{alg:tau_and_m}.

\begin{algorithm}
  \caption{Computing the linear combination~\eqref{eq:lincomb}.}
  \label{alg:mainalg}
  \begin{algorithmic}[0]
    \STATE $t = 0; \, k = 0; \, u_k = b_0$
    \STATE Evaluate initial $\tau$ using \eqref{eq:tau_0}; 
    \STATE Initial guess $m = 10$
    \REPEAT
    \STATE Compute $w_0,\ldots,w_p$ according to \eqref{eq:wrecurrence}
    \REPEAT
    \STATE Compute $H_m$ and $B_m$ using Algorithm~\ref{alg:arnoldi}
    or~\ref{alg:lanczos}
    \STATE $F = $ approximation to $\varphi_p(\tau A) w_p$ given by
    \eqref{eq:phipapprox2} 
    \STATE $\varepsilon = $ error estimate given by \eqref{eq:errest}
    \STATE Compute $\omega$ according to~\eqref{eq:taunew}
    \STATE Compute $\taunew$ and $\mnew$ using Algorithm~\ref{alg:tau_and_m}
    \STATE Compute $C(\taunew,m)$ and $C(\tau,\mnew)$ according
    to~\eqref{eq:8}, \eqref{eq:4} and \eqref{eq:5}
    \IF{$C(\taunew,m) < C(\tau,\mnew)$}
    \STATE $\tau = \min \bigl\{ \max \bigl\{ \taunew, \tfrac15\tau
    \bigr\}, 2\tau, 1-t \bigr\}$ 
    \ELSE
    \STATE $m = \min \bigl\{ \max \bigl\{ \mnew,
    \lfloor\tfrac34m\rfloor, 1 \bigr\}, \lceil\tfrac43m\rceil \bigr\}$ 
    \ENDIF
    \UNTIL $\omega \le \delta$
    \STATE Compute $u_{k+1}$ according to~\eqref{eq:step}
    \STATE $t = t + \tau; \, k = k + 1$
    \UNTIL $t = 1$
    \RETURN $u_k$
  \end{algorithmic}
\end{algorithm}

\begin{algorithm}
  \caption{Computing $\taunew$ and $\mnew$.}
  \label{alg:tau_and_m}
  \begin{algorithmic}[0]
    \IF{previous step was rejected and $\tau$ was reduced}
    \STATE Compute $\hat{q}$ according to~\eqref{eq:6}
    \ELSIF{previous step was rejected and $\hat{q}$ was computed in
      previous step} 
    \STATE Keep old value of $\hat{q}$
    \ELSE
    \STATE $\hat{q} = \tfrac14m$
    \ENDIF
    \STATE Compute $\taunew$ according to~\eqref{eq:2}
    \IF{previous step was rejected and $m$ was reduced}
    \STATE Compute $\hat\kappa$ according to~\eqref{eq:7}
    \ELSIF{previous step was rejected and $\hat\kappa$ was computed
      in previous step}
    \STATE Keep old value of $\hat\kappa$
    \ELSE
    \STATE $\hat\kappa = 2$
    \ENDIF
    \STATE Compute $\mnew$ according to~\eqref{eq:3}
  \end{algorithmic}
\end{algorithm}

\subsection{The \matlab\ function \phipm}
\label{sect:matlab}

The algorithm described above is implemented in a \matlab\ function
called \phipm. In terms of computing power, the system requirements
are modest.  Any computer capable of running a moderately up-to-date
version of {\matlab} is sufficient. 

A call of the \phipm\ function has the form
\begin{verbatim}
  [u, stats] = phipm(t, A, b, tol, symm, m)
\end{verbatim}
There are three mandatory input arguments and one mandatory output
argument; the other arguments are optional. The first input argument
is~\texttt{t}, the final time, $\tend$, for the differential
equation~\eqref{eq:phide}. This is generally chosen to be $t=1$ because
the solution~\eqref{eq:lincomb2} of~\eqref{eq:phide} at $t=1$ equals
the linear combination~\eqref{eq:lincomb}. The second argument is the
$n$-by-$n$ matrix argument of the $\varphi$-functions. The \phipm\
function can also be used without forming the matrix~$A$ explicitly,
by setting the argument \texttt{A} to a function which, given a
vector~$b$, computes~$Ab$. Finally, \texttt{b}~is an $n$-by-$(p+1)$
matrix with columns representing the vectors $b_0,b_1,\ldots,b_p$ to
be multiplied by the corresponding $\varphi$-functions. There is one
mandatory output argument: \texttt{u}, the numerical approximation to
the solution of~\eqref{eq:phide} at the final time~$\tend$.

There are three optional input arguments. The first one
is~\texttt{tol}, the tolerance~Tol in~\eqref{eq:6}. The default
tolerance is~$10^{-7}$.  Then comes~\texttt{symm}, a boolean
indicating whether $A$ is symmetric (\texttt{symm=1}) or not
(\texttt{symm=0}). If not supplied, the code determines itself whether
$A$ is symmetric if the matrix is passed explicitly, and assumes that
$A$ is not symmetric if the matrix is given implicitly. The final
input argument is \texttt{m}, the initial choice for the dimension of
the Krylov subspace. The initial choice is $m=1$ by default. There is
also one optional output argument: \texttt{stats}, for providing the
user with various statistics of the computation. It is a vector with
four entries: \texttt{stats(1)} is the number of steps needed to
complete the integration, \texttt{stats(2)} is the number of rejected
steps, \texttt{stats(3)} is the number of matrix-vector products, and
\texttt{stats(4)} is the number of matrix exponentials computed.


\section{Numerical experiments}
\label{sect:numer}

In this section we perform several numerical experiments, which
illustrate the advantages of the approach that has been outlined in
the previous sections. We compare the function \phipm\ with various
state-of-the-art numerical algorithms. The first experiment compares
several numerical ODE solvers on a large system of linear ODEs,
resulting from the finite-difference discretization of the Heston PDE,
a common example from the mathematical finance literature. The second
experiment compares our function \phipm\ and the \expv\ and \phiv\
functions from \expokit\ on various large sparse matrices. This
repeats the experiment reported by~\citeN{sidje98esp}.

All experiments use a MacBookPro with 2.66 GHz Intel Core 2 Duo
processor and 4GB 1067 MHz DDR3 memory. We use \matlab\ Version~7.11
(R2010b) for all experiments, which computes the matrix exponential as
described in \citeN{higham05tsa}. We have noticed that the results
described in this section depend on the specifications of the computer
but the overall nature of the numerical experiments remains the same.

\subsection{The Heston equation in financial mathematics}
\label{sect:heston}

A European call option gives its owner the right (but not obligation)
to buy a certain asset for a certain price (called the strike price)
at a certain time (the expiration date). European call options with
stochastic volatility, modeled by a stochastic mean-reverting
differential equation, have been successfully priced
by~\citeN{heston93acf}. The Heston pricing formulae are a natural
extension of the celebrated Black--Scholes--Merton pricing
formulae. Despite the existence of a semi-closed form solution, which
requires the numerical computation of an indefinite integral, the
so-called Heston PDE is often used as a test example to compare
various numerical integrators. This example follows
closely~\citeN{inthout07asi} and~\citeN{inthout10afd}.

Let $U(s,v,t)$ denote the European call option price at time $T-t$,
where $s$ is the price of the underlying asset and $v$ the variance in
the asset price at that time. Here, $T$ is the expiration date of the
option. Heston's stochastic volatility model ensures that the price of
a European call option satisfies the time-dependent
convection-diffusion-reaction equation
\begin{equation*}
  U_t = \frac{1}{2} v s^2 U_{ss} + \rho \lambda v s U_{vs} +
  \frac{1}{2} \lambda^2 v U_{vv} + (r_d - r_f) s U_s + \kappa (\eta - v) U_v
  - r_d U.
\end{equation*}
This equation is posed on the unbounded spatial domain $s > 0$ and $v
\geq 0$, while $t$ ranges from 0 to~$T$. Here, $\rho \in [-1,1]$
represents the correlation between the Wiener processes modeling the
asset price and its variance, $\lambda$ is a positive scaling
constant, $r_d$ and $r_f$ are constants representing the risk-neutral
domestic and foreign interest rates respectively, $\eta$ 
represents the mean level of $v$ and $\kappa$ the rate at which $v$
reverts to $\eta$. The payoff of the European call option provides the
initial condition
\begin{equation*}
  U(s,v,0) = \max(s-K,0),
\end{equation*}
where $K \geq 0$ denotes the strike price. 

The unbounded spatial domain must be restricted in size in
computations and we choose the sufficiently large rectangle $[0, S]
\times [0, V]$. Usually $S$ and $V$ are chosen much larger than the
values of $s$ and $v$ of practical interest. This is a commonly used
approach in financial modeling so that if the boundary conditions are
imperfect their effect is minimized, see
\citeN[p.~121]{tavella00pfi}.  Suitable boundary conditions for $0 <
t \leq T$ are
\begin{align*}
  & U(0,v,t) = 0, \\
  & U(s,V,t) = s, \\
  & U_s(S,v,t) = 1, \\
  & U_t(s,0,t) - (r_d - r_f) s U_s(s,0,t) - \kappa \eta U_v(s,0,t) + r U(s,0,t) =
  0. 
\end{align*} 
Note that the boundary and initial conditions are inconsistent or
non-matching, that is the boundary and initial conditions at $t=0$ do
not agree.  

We discretize the spatial domain using a uniform rectangular mesh with
mesh lengths $\Delta s$ and $\Delta v$ and use standard second-order
finite differences to approximate the derivatives as follows:
{\allowdisplaybreaks
\begin{align*}
  (U_s)_{i,j} &\approx \frac{U_{i+1,j}-U_{i-1,j}}{2 \Delta s}, \\
  (U_{ss})_{i,j} &\approx \frac{U_{i+1,j}-2U_{i,j}+U_{i-1,j}}{\Delta
    s^2}, \\
  (U_v)_{i,j} &\approx \frac{U_{i,j+1}-U_{i,j-1}}{2 \Delta v}, \\
  (U_v)_{i,0} &\approx \frac{-3U_{i,0}+4U_{i,1}-U_{i,2}}{2 \Delta v}, \\
  (U_{vv})_{i,j} &\approx \frac{U_{i,j+1}-2U_{i,j}+U_{i,j-1}}{\Delta
    v^2}, \\
  (U_{sv})_{i,j} &\approx
  \frac{U_{i+1,j+1}+U_{i-1,j-1}-U_{i-1,j+1}-U_{i+1,j-1}}{4 \Delta v
    \Delta s}.
\end{align*}
}
The boundary associated to $v=0$ is included in the mesh, but the
other three boundaries are not.  Combining the finite-difference
discretization and the boundary conditions leads to a large system of
ODEs of the form
\begin{equation*}
  u'(t) = A u(t) + b_1, \qquad u(0) = b_0.
\end{equation*}
The exact solution for this system is given in
Lemma~\ref{lem:exactsolution}. This is a natural problem for the
\phipm\ solver.

We compare five methods: the scheme of \citeN{crank47apm}, two
Alternating Direction Implicit (ADI) schemes, \odeofs\ from \matlab\
and the \phipm\ method described in this paper. The two ADI schemes
are the method due to \citeN{douglas56otn} with $\theta=\frac{1}{2}$,
and the method of~\citeN{hundsdorfer03nso} with $\theta=\frac{3}{10}$
and $\mu=\frac{1}{2}$. The first ADI method is of order one and the
second method is of order two. \citeN{inthout07asi} lists each of the ADI
methods described above and explains necessary implementation details;
the stability of these methods is discussed in \citeN{inthout08uso}.
To compare \phipm\ with an adaptive solver we choose the in-built
\matlab\ solver \odeofs, which is a variable step size, variable
order implementation of the backward differentiation formulae (BDF).
The Jacobian of the right-hand side (that is, the matrix~$A$) is
passed to \odeofs; this makes \odeofs\ considerably faster.

We choose the same problem parameters as in the paper by
\cite{inthout07asi}: namely $\kappa=2$, $\nu=0.2$, $\lambda=0.3$,
$\rho=0.8$, $r_d=0.03$, $r_f=0.0$, the option maturity $T=1$ and the
strike price $K=100$. The spatial domain is truncated to $[0, 8K]
\times [0, 5]$. We use a grid with 100~points in the $s$-direction and
51~points in the $v$-direction (recall that $v=0$ is included in the
grid and the other borders are not). This results in a system with
5100~degrees of freedom, with $nnz=44,800$ non-zero elements.
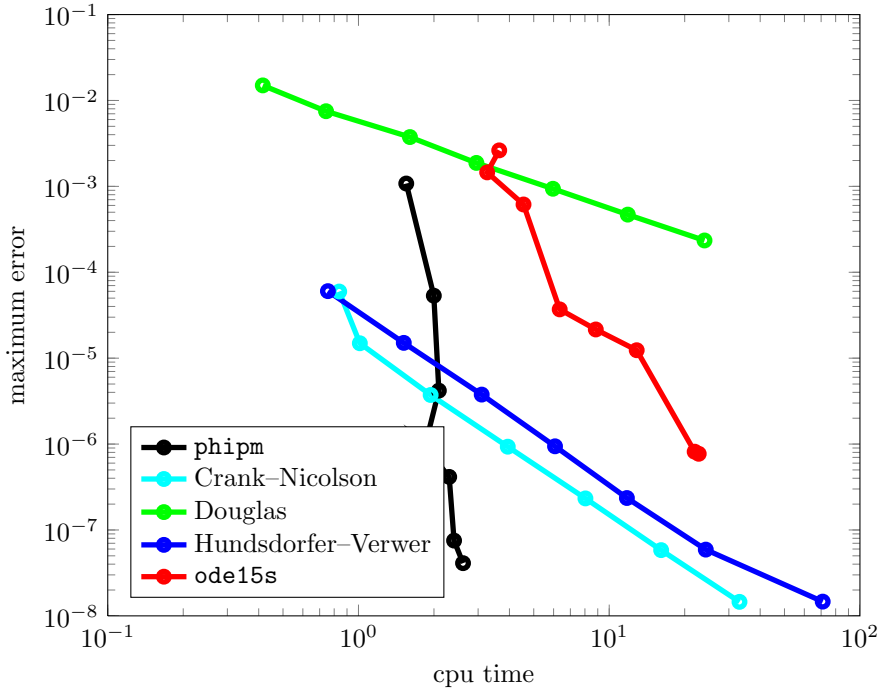
\begin{figure}[!ht]
  \centering
%
%
\begin{tikzpicture}

\definecolor{mycolor1}{rgb}{0,1,1}

\begin{loglogaxis}[%
view={0}{90},
scale only axis,
width=10cm,
height=8cm,
xmin=0.1, xmax=100,
ymin=1e-08, ymax=0.1,
xlabel={cpu time},
ylabel={maximum error},
axis on top,
legend entries={\phipm,Crank--Nicolson,Douglas,Hundsdorfer--Verwer,\texttt{ode15s}},
legend style={at={(0.03,0.03)},anchor=south west,nodes=right}]
\addplot [
color=black,
solid,
line width=2.0pt,
mark=o,
mark options={solid}
]
coordinates{ (1.55077,0.00107945) (1.99768,5.33854e-05) (2.0919,4.18875e-06) (1.85109,1.00292e-06) (1.7339,1.15505e-06) (2.3005,4.14903e-07) (2.40563,7.52279e-08) (2.612,4.10747e-08)
};

\addplot [
color=mycolor1,
solid,
line width=2.0pt,
mark=o,
mark options={solid}
]
coordinates{ (0.837858,5.96555e-05) (1.0126,1.49139e-05) (1.94022,3.72848e-06) (3.93932,9.3212e-07) (8.03731,2.33023e-07) (16.1132,5.82568e-08) (33.0858,1.45607e-08)
};

\addplot [
color=green,
solid,
line width=2.0pt,
mark=o,
mark options={solid}
]
coordinates{ (0.415302,0.0149806) (0.742444,0.00749779) (1.60421,0.00375078) (2.95291,0.00187586) (5.96359,0.00093805) (11.8601,0.000469054) (23.993,0.000234535)
};

\addplot [
color=blue,
solid,
line width=2.0pt,
mark=o,
mark options={solid}
]
coordinates{ (0.755429,6.03039e-05) (1.51611,1.50726e-05) (3.10091,3.76769e-06) (6.0796,9.41858e-07) (11.7633,2.35518e-07) (24.273,5.90694e-08) (71.036,1.46546e-08)
};

\addplot [
color=red,
solid,
line width=2.0pt,
mark=o,
mark options={solid}
]
coordinates{ (3.64214,0.00262518) (3.26041,0.0014551) (4.55264,0.000615536) (6.34644,3.69866e-05) (8.84033,2.15846e-05) (12.8717,1.23718e-05) (21.9167,8.15551e-07) (22.7508,7.70858e-07)
};

\end{loglogaxis}
\end{tikzpicture}
 \caption{Plots of maximum error (on the domain $[0,2K] \times [0,1]$)
   against CPU time for the system of ODEs from the discretized Heston
   PDE. The two ADI schemes are the represented by the: green line
   (Douglas); blue line (Hundsdorfer and Verwer), the cyan line is
   Crank--Nicolson, the red line is \odeofs\ and the black line is
   \phipm. \label{fig:ex1}}
\end{figure}
Figure~\ref{fig:ex1} shows the error of the solvers against the CPU
time. The Crank--Nicolson and ADI schemes are run with the step size
decreasing in powers of two from~$2^{-8}$ to~$2^{-14}$, while \phipm\
and \odeofs\ are run with the tolerance decreasing geometrically
from~$10^{-1}$ to~$10^{-6}$.  The error is computed by comparing the
numerical solution against the ``exact'' solution, as computed using
two different methods with very small stepsizes so the the solution
was accurate to within~$10^{-10}$. We measure the maximum error at
time $t=T$ of the numerical solution satisfying $[0,2K] \times [0,1]$,
a smaller domain than the computational domain.

The first surprising result is that the Crank--Nicolson method
outperforms the ADI methods. This requires the use of column
reordering and row scaling in the LU decomposition as implemented in
\matlab. A call to this function takes the form \texttt{[L,U,P,Q,R] =
  lu(X)}. Practitioners are interested in this problem for accuracy
levels of around $10^{-4}$ or one basis point. The \phipm\ method is
the most efficient for an accuracy of around $10^{-6}$. Similar
results hold for the four parameter sets listed in
\citeN{inthout10afd}, with the cross-over point close
to~$10^{-6}$. Note that for these four parameter sets the $U_v$ term
was discretized using upwinding when $v>1$. Recently, we have applied
Krylov subspace methods to a variety of option pricing problems. We
find that Krylov subspace methods significantly outperform ADI methods
for dimension higher than two; we refer the interested reader to the
forthcoming paper \cite{niesen11ksm}.

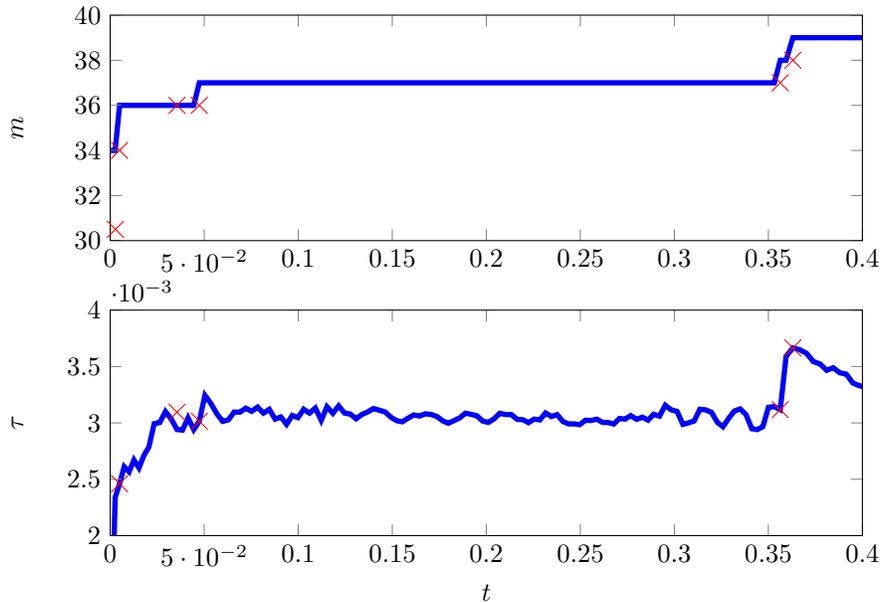
\begin{figure}
  \centering
%
%
\begin{tikzpicture}

\begin{axis}[%
view={0}{90},
name=plot1,
scale only axis,
width=10cm,
height=3cm,
xmin=0, xmax=0.4,
ymin=30, ymax=40,
ylabel={$m$},
axis on top]
\addplot [
color=blue,
solid,
line width=2.0pt
]
coordinates{ (0,34) (0.00121169,34) (0.00264517,34) (0.00498362,36) (0.00744394,36) (0.0100553,36) (0.0126213,36) (0.0152923,36) (0.0178931,36) (0.0206028,36) (0.0233902,36) (0.0263832,36) (0.0293864,36) (0.0324868,36) (0.0355215,36) (0.0384622,36) (0.0413982,36) (0.0444425,36) (0.0473825,37) (0.0503972,37) (0.0536416,37) (0.0568113,37) (0.0598898,37) (0.0629031,37) (0.0659303,37) (0.0690262,37) (0.0721217,37) (0.0752519,37) (0.0783545,37) (0.0814946,37) (0.0845814,37) (0.0877016,37) (0.0907351,37) (0.0937878,37) (0.0967753,37) (0.0998411,37) (0.102888,37) (0.106009,37) (0.109095,37) (0.112227,37) (0.115249,37) (0.118393,37) (0.121478,37) (0.124628,37) (0.127715,37) (0.130794,37) (0.13383,37) (0.136904,37) (0.140001,37) (0.143129,37) (0.146241,37) (0.149338,37) (0.152387,37) (0.155406,37) (0.158416,37) (0.161458,37) (0.164528,37) (0.16759,37) (0.17067,37) (0.173745,37) (0.176801,37) (0.17982,37) (0.182816,37) (0.185836,37) (0.18888,37) (0.191966,37) (0.195042,37) (0.198104,37) (0.20112,37) (0.204125,37) (0.20716,37) (0.210246,37) (0.21332,37) (0.216394,37) (0.219425,37) (0.222453,37) (0.225456,37) (0.228489,37) (0.231516,37) (0.234604,37) (0.237661,37) (0.240734,37) (0.243746,37) (0.246737,37) (0.249729,37) (0.252713,37) (0.255736,37) (0.258758,37) (0.261791,37) (0.264795,37) (0.267798,37) (0.270788,37) (0.273827,37) (0.276861,37) (0.279926,37) (0.282957,37) (0.286007,37) (0.289031,37) (0.292108,37) (0.295168,37) (0.298325,37) (0.30144,37) (0.304542,37) (0.307529,37) (0.310529,37) (0.313546,37) (0.316667,37) (0.319782,37) (0.322875,37) (0.32588,37) (0.328846,37) (0.331885,37) (0.33499,37) (0.338113,37) (0.341187,37) (0.344135,37) (0.347076,37) (0.350042,37) (0.353181,37) (0.356325,38) (0.359443,38) (0.363033,39) (0.366698,39) (0.370347,39) (0.373963,39) (0.377506,39) (0.381029,39) (0.384495,39) (0.387984,39) (0.391428,39) (0.394861,39) (0.398215,39) (0.401546,39)
};

\addplot [
color=red,
mark size=4.5pt,
only marks,
mark=x,
mark options={solid}
]
coordinates{ (0.00280329,30.5) (0.00498362,34) (0.0355215,36) (0.0473825,36) (0.356325,37) (0.363033,38) (0.646311,39)
};

\end{axis}

\begin{axis}[%
view={0}{90},
at=(plot1.below south west), anchor=above north west,
scale only axis,
width=10cm,
height=3cm,
xmin=0, xmax=0.4,
ymin=0.002, ymax=0.004,
xlabel={$t$},
ylabel={$\tau$},
axis on top]
\addplot [
color=blue,
solid,
line width=2.0pt
]
coordinates{ (0.00121169,0.00143349) (0.00264517,0.00233845) (0.00498362,0.00246032) (0.00744394,0.00261133) (0.0100553,0.00256602) (0.0126213,0.00267104) (0.0152923,0.00260075) (0.0178931,0.00270974) (0.0206028,0.00278742) (0.0233902,0.00299294) (0.0263832,0.00300324) (0.0293864,0.00310036) (0.0324868,0.00303468) (0.0355215,0.00294079) (0.0384622,0.00293595) (0.0413982,0.00304425) (0.0444425,0.00294009) (0.0473825,0.00301466) (0.0503972,0.00324436) (0.0536416,0.0031697) (0.0568113,0.0030785) (0.0598898,0.00301335) (0.0629031,0.0030272) (0.0659303,0.0030959) (0.0690262,0.00309543) (0.0721217,0.00313026) (0.0752519,0.00310255) (0.0783545,0.00314017) (0.0814946,0.00308675) (0.0845814,0.00312026) (0.0877016,0.00303342) (0.0907351,0.0030527) (0.0937878,0.0029875) (0.0967753,0.00306586) (0.0998411,0.00304671) (0.102888,0.00312092) (0.106009,0.0030863) (0.109095,0.00313184) (0.112227,0.0030225) (0.115249,0.00314347) (0.118393,0.00308479) (0.121478,0.00315075) (0.124628,0.00308697) (0.127715,0.0030784) (0.130794,0.00303591) (0.13383,0.00307456) (0.136904,0.00309698) (0.140001,0.00312756) (0.143129,0.00311255) (0.146241,0.00309628) (0.149338,0.00304943) (0.152387,0.00301928) (0.155406,0.00300991) (0.158416,0.00304205) (0.161458,0.00307011) (0.164528,0.00306116) (0.16759,0.00308059) (0.17067,0.00307438) (0.173745,0.00305666) (0.176801,0.00301851) (0.17982,0.0029967) (0.182816,0.00301966) (0.185836,0.00304433) (0.18888,0.00308576) (0.191966,0.00307587) (0.195042,0.00306236) (0.198104,0.00301599) (0.20112,0.00300434) (0.204125,0.00303572) (0.20716,0.00308522) (0.210246,0.00307411) (0.21332,0.00307411) (0.216394,0.00303134) (0.219425,0.00302825) (0.222453,0.00300234) (0.225456,0.00303351) (0.228489,0.00302626) (0.231516,0.00308863) (0.234604,0.00305638) (0.237661,0.0030734) (0.240734,0.0030123) (0.243746,0.00299118) (0.246737,0.00299118) (0.249729,0.00298447) (0.252713,0.00302326) (0.255736,0.00302118) (0.258758,0.00303307) (0.261791,0.0030045) (0.264795,0.00300289) (0.267798,0.00299016) (0.270788,0.00303917) (0.273827,0.00303358) (0.276861,0.00306466) (0.279926,0.00303107) (0.282957,0.00305046) (0.286007,0.00302427) (0.289031,0.00307654) (0.292108,0.00306026) (0.295168,0.00315649) (0.298325,0.00311554) (0.30144,0.0031019) (0.304542,0.00298672) (0.307529,0.0030004) (0.310529,0.00301673) (0.313546,0.00312087) (0.316667,0.00311519) (0.319782,0.00309347) (0.322875,0.00300454) (0.32588,0.00296557) (0.328846,0.00303908) (0.331885,0.00310511) (0.33499,0.00312358) (0.338113,0.00307364) (0.341187,0.00294842) (0.344135,0.00294051) (0.347076,0.00296649) (0.350042,0.00313833) (0.353181,0.00314437) (0.356325,0.00311766) (0.359443,0.00359044) (0.363033,0.00366474) (0.366698,0.00364901) (0.370347,0.00361631) (0.373963,0.00354313) (0.377506,0.00352256) (0.381029,0.00346582) (0.384495,0.00348921) (0.387984,0.00344424) (0.391428,0.00343251) (0.394861,0.00335444) (0.398215,0.00333048) (0.401546,0.00331622)
};

\addplot [
color=red,
mark size=4.5pt,
only marks,
mark=x,
mark options={solid}
]
coordinates{ (0,0.00121169) (0.00498362,0.00246032) (0.0355215,0.00309423) (0.0473825,0.00301466) (0.356325,0.00311766) (0.363033,0.00366474) (0.646311,0.00358813)
};

\end{axis}
\end{tikzpicture}
  \caption{Plots showing how the dimension $m$ of the Krylov subspace
    and the step size~$\tau$ change during the integration of the
    system of ODEs from the discretized Heston PDE, as solved by
    \phipm\ with a tolerance of~$10^{-4}$. The red crosses represent
    rejected steps. \label{fig:ex1hist}}
\end{figure}

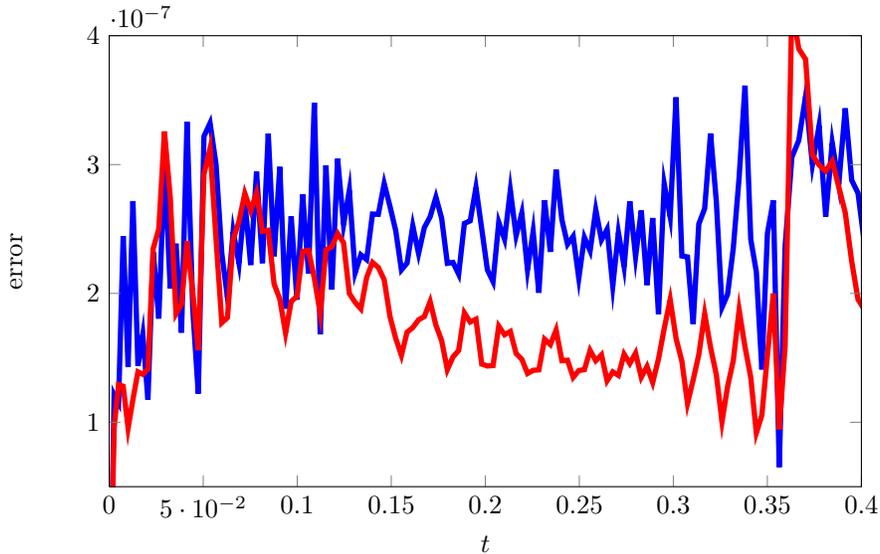
\begin{figure}
  \centering
%
%
\begin{tikzpicture}

\begin{axis}[%
view={0}{90},
scale only axis,
width=10cm,
height=6cm,
xmin=0, xmax=0.4,
ymin=5e-08, ymax=4e-07,
xlabel={$t$},
ylabel={error},
axis on top]
\addplot [
color=blue,
solid,
line width=2.0pt
]
coordinates{ (0,2.32251e-08) (0.00121169,1.79037e-09) (0.00264517,1.21471e-07) (0.00498362,1.15146e-07) (0.00744394,2.44547e-07) (0.0100553,1.43078e-07) (0.0126213,2.71646e-07) (0.0152923,1.43789e-07) (0.0178931,1.68091e-07) (0.0206028,1.17552e-07) (0.0233902,2.3215e-07) (0.0263832,1.80419e-07) (0.0293864,3.00743e-07) (0.0324868,2.03815e-07) (0.0355215,2.3866e-07) (0.0384622,1.69535e-07) (0.0413982,3.33153e-07) (0.0444425,1.87738e-07) (0.0473825,1.22276e-07) (0.0503972,3.21918e-07) (0.0536416,3.32189e-07) (0.0568113,3.00169e-07) (0.0598898,2.31052e-07) (0.0629031,1.96782e-07) (0.0659303,2.48019e-07) (0.0690262,2.2329e-07) (0.0721217,2.71884e-07) (0.0752519,2.2202e-07) (0.0783545,2.94421e-07) (0.0814946,2.23472e-07) (0.0845814,3.24086e-07) (0.0877016,2.28864e-07) (0.0907351,2.98194e-07) (0.0937878,1.88102e-07) (0.0967753,2.59902e-07) (0.0998411,1.95094e-07) (0.102888,2.76812e-07) (0.106009,2.15622e-07) (0.109095,3.4805e-07) (0.112227,1.68191e-07) (0.115249,2.99368e-07) (0.118393,2.02916e-07) (0.121478,3.04552e-07) (0.124628,2.53393e-07) (0.127715,2.80059e-07) (0.130794,2.16052e-07) (0.13383,2.29977e-07) (0.136904,2.26236e-07) (0.140001,2.61587e-07) (0.143129,2.61374e-07) (0.146241,2.85221e-07) (0.149338,2.6744e-07) (0.152387,2.48583e-07) (0.155406,2.18259e-07) (0.158416,2.23553e-07) (0.161458,2.52329e-07) (0.164528,2.30972e-07) (0.16759,2.5109e-07) (0.17067,2.59457e-07) (0.173745,2.74662e-07) (0.176801,2.58231e-07) (0.17982,2.23396e-07) (0.182816,2.24059e-07) (0.185836,2.14917e-07) (0.18888,2.54303e-07) (0.191966,2.56297e-07) (0.195042,2.82125e-07) (0.198104,2.50075e-07) (0.20112,2.18319e-07) (0.204125,2.09113e-07) (0.20716,2.55192e-07) (0.210246,2.42089e-07) (0.21332,2.79955e-07) (0.216394,2.44809e-07) (0.219425,2.62299e-07) (0.222453,2.18305e-07) (0.225456,2.48107e-07) (0.228489,2.00471e-07) (0.231516,2.72286e-07) (0.234604,2.32268e-07) (0.237661,2.96058e-07) (0.240734,2.57189e-07) (0.243746,2.38094e-07) (0.246737,2.44315e-07) (0.249729,2.11872e-07) (0.252713,2.43407e-07) (0.255736,2.33075e-07) (0.258758,2.64844e-07) (0.261791,2.41551e-07) (0.264795,2.49861e-07) (0.267798,2.05813e-07) (0.270788,2.47311e-07) (0.273827,2.2085e-07) (0.276861,2.71487e-07) (0.279926,2.28591e-07) (0.282957,2.643e-07) (0.286007,2.06479e-07) (0.289031,2.58501e-07) (0.292108,1.83852e-07) (0.295168,2.84941e-07) (0.298325,2.59565e-07) (0.30144,3.52153e-07) (0.304542,2.29047e-07) (0.307529,2.28283e-07) (0.310529,1.76315e-07) (0.313546,2.53911e-07) (0.316667,2.6587e-07) (0.319782,3.24127e-07) (0.322875,2.71221e-07) (0.32588,1.89163e-07) (0.328846,1.99284e-07) (0.331885,2.35149e-07) (0.33499,2.90063e-07) (0.338113,3.61272e-07) (0.341187,2.41812e-07) (0.344135,2.16853e-07) (0.347076,1.40968e-07) (0.350042,2.46643e-07) (0.353181,2.72203e-07) (0.356325,6.52199e-08) (0.359443,2.36451e-07) (0.363033,3.05733e-07) (0.366698,3.18701e-07) (0.370347,3.53121e-07) (0.373963,3.00006e-07) (0.377506,3.30146e-07) (0.381029,2.59663e-07) (0.384495,3.16775e-07) (0.387984,2.8486e-07) (0.391428,3.43651e-07) (0.394861,2.87784e-07) (0.398215,2.77818e-07) (0.401546,2.39417e-07)
};

\addplot [
color=red,
solid,
line width=2.0pt
]
coordinates{ (0,2.1527e-08) (0.00121169,1.62436e-09) (0.00264517,9.58928e-08) (0.00498362,1.29218e-07) (0.00744394,1.27923e-07) (0.0100553,9.54734e-08) (0.0126213,1.17747e-07) (0.0152923,1.39112e-07) (0.0178931,1.37203e-07) (0.0206028,1.41253e-07) (0.0233902,2.34753e-07) (0.0263832,2.50116e-07) (0.0293864,3.25736e-07) (0.0324868,2.74297e-07) (0.0355215,1.8419e-07) (0.0384622,1.91696e-07) (0.0413982,2.40479e-07) (0.0444425,2.0058e-07) (0.0473825,1.55825e-07) (0.0503972,2.92633e-07) (0.0536416,3.12526e-07) (0.0568113,2.42441e-07) (0.0598898,1.77195e-07) (0.0629031,1.81308e-07) (0.0659303,2.4171e-07) (0.0690262,2.55882e-07) (0.0721217,2.77591e-07) (0.0752519,2.64238e-07) (0.0783545,2.77967e-07) (0.0814946,2.48181e-07) (0.0845814,2.48986e-07) (0.0877016,2.07235e-07) (0.0907351,1.95626e-07) (0.0937878,1.69246e-07) (0.0967753,1.93739e-07) (0.0998411,1.9778e-07) (0.102888,2.32568e-07) (0.106009,2.33416e-07) (0.109095,2.11843e-07) (0.112227,1.84385e-07) (0.115249,2.337e-07) (0.118393,2.3631e-07) (0.121478,2.46453e-07) (0.124628,2.39343e-07) (0.127715,1.99941e-07) (0.130794,1.93053e-07) (0.13383,1.88012e-07) (0.136904,2.12346e-07) (0.140001,2.24083e-07) (0.143129,2.20163e-07) (0.146241,2.10629e-07) (0.149338,1.81969e-07) (0.152387,1.65007e-07) (0.155406,1.51937e-07) (0.158416,1.69947e-07) (0.161458,1.73456e-07) (0.164528,1.79544e-07) (0.16759,1.81987e-07) (0.17067,1.93652e-07) (0.173745,1.75363e-07) (0.176801,1.63301e-07) (0.17982,1.40834e-07) (0.182816,1.50655e-07) (0.185836,1.55706e-07) (0.18888,1.849e-07) (0.191966,1.77825e-07) (0.195042,1.7996e-07) (0.198104,1.4514e-07) (0.20112,1.43667e-07) (0.204125,1.44016e-07) (0.20716,1.74546e-07) (0.210246,1.67932e-07) (0.21332,1.70456e-07) (0.216394,1.53278e-07) (0.219425,1.48923e-07) (0.222453,1.37865e-07) (0.225456,1.4031e-07) (0.228489,1.40674e-07) (0.231516,1.6451e-07) (0.234604,1.6029e-07) (0.237661,1.7161e-07) (0.240734,1.4777e-07) (0.243746,1.48081e-07) (0.246737,1.34831e-07) (0.249729,1.40176e-07) (0.252713,1.40816e-07) (0.255736,1.55794e-07) (0.258758,1.47765e-07) (0.261791,1.52335e-07) (0.264795,1.32198e-07) (0.267798,1.39138e-07) (0.270788,1.36384e-07) (0.273827,1.52864e-07) (0.276861,1.4561e-07) (0.279926,1.53638e-07) (0.282957,1.34931e-07) (0.286007,1.43298e-07) (0.289031,1.3092e-07) (0.292108,1.49133e-07) (0.295168,1.72961e-07) (0.298325,1.95134e-07) (0.30144,1.6509e-07) (0.304542,1.47248e-07) (0.307529,1.14671e-07) (0.310529,1.32813e-07) (0.313546,1.52461e-07) (0.316667,1.84128e-07) (0.319782,1.58122e-07) (0.322875,1.36728e-07) (0.32588,1.00517e-07) (0.328846,1.27649e-07) (0.331885,1.47877e-07) (0.33499,1.87141e-07) (0.338113,1.58447e-07) (0.341187,1.34695e-07) (0.344135,9.17601e-08) (0.347076,1.05445e-07) (0.350042,1.52037e-07) (0.353181,1.99958e-07) (0.356325,9.45797e-08) (0.359443,1.59298e-07) (0.363033,4.22398e-07) (0.366698,3.89932e-07) (0.370347,3.81651e-07) (0.373963,3.07406e-07) (0.377506,2.99617e-07) (0.381029,2.9498e-07) (0.384495,3.0269e-07) (0.387984,2.82524e-07) (0.391428,2.62663e-07) (0.394861,2.24784e-07) (0.398215,1.95373e-07) (0.401546,1.87537e-07)
};

\end{axis}
\end{tikzpicture}
  \caption{The error estimate (blue) and the actual error (red)
    computed during the integration of the  system of ODEs from the
    discretized Heston PDE. Only accepted steps are shown. Again, the
    tolerance is~$10^{-4}$. \label{fig:ex1error}}
\end{figure}

Figure~\ref{fig:ex1hist} shows how the
Krylov subspace size (top graph) and the step size (bottom graph)
change during the integration interval. They both vary during the
integration, which shows that the adaptivity presented in
Section~\ref{sect:adaptive} is effective. Figure~\ref{fig:ex1error}
plots the error estimate and the actual error during the integration
interval. The error estimate is always larger than the actual error in
this experiment. 

One worrying aspect in these figures is that the step size sequence in
Figure~\ref{fig:ex1hist} zig-zags and that the estimated error in
Figure~\ref{fig:ex1error} varies alot from step to step. This might
indicate stability problems, perhaps caused by augmenting the matrix
in~(\ref{eq:hathm}) or by the matrix-vector multiplications
in~(\ref{eq:wrecurrence}). The \matlab\ solver \odeofs\ also has a
rather strange stepsize pattern for this problem, where stepsizes are
constant or rapidly increasing. The error behaviour of Krylov methods
is a difficult problem and we intend to investigate this further in
future work; perhaps the use of control theory techniques can be
useful in smoothing the error estimate and choice of Krylov subspace
size.  Also, given that this equation has a semi-closed form solution,
the application to more exotic options, such as barrier, American and
Asian options, is of more practical significance and in our minds for
future work.

Finally, we draw attention to the fact that the linear algebra in
\odeofs\, the Crank--Nicolson and all the ADI methods, is performed
using highly optimized routines written in a low-level language,
whereas all the computations in \phipm\ are done in the native
\matlab\ language. We have noticed that the Crank--Nicolson and the
ADI methods performance has improved relative to \phipm\ in newer
version of \matlab. We intend to release C++ and CUDA versions of this
software in the near future.

\subsection{Comparisons between \phipm\ and the functions in \expokit}
\label{sect:expokit}

The University of Florida sparse matrix collection, compiled by
\citeN{davis08uof}, is an excellent, well-maintained website
containing various classes of sparse matrices. All the sparse matrices
that we use in this subsection are available from that website.

\subsubsection*{Experiment 1}
 
We compute the action of the matrix exponential of four different
sparse matrices described below on certain vectors $b_0$. We use the
\expv\ and \phiv\ routines from \expokit\ and the \phipm\ routine
described in this paper to compute $\e^{tA}b_0$ (the \phiv\ routine
computes $\e^{tA}b_0 + \varphi_1(tA)b_1$, so by setting $b_1=0$ this
code can also be used to compute $\e^{tA}b_0$). The \expv\ and \phiv\
routines are implemented in \matlab\ like the \phipm\ routine, but
they do not use the Lanczos algorithm when $A$ is symmetric.
Therefore, we implemented a variant of \phipm\ which computes
\eqref{eq:lincomb2} using a Krylov subspace of fixed dimension $m=30$.
This variant is called \phip. It can be considered as an extension of
the \phiv\ code to symmetric matrices and $p \ge 1$ and similar to the
code developed by \citeN{sofroniou07eco}.

The four matrices we consider are as follows:
\begin{itemize}
\item The first matrix, \texttt{orani678} from the Harwell--Boeing
  collection \cite{duff89smt}, is an unsymmetric sparse matrix of
  order $n=2,529$ with $nnz=90,158$ nonzero elements. We choose
  $t=10$, $b_0=[1,1,\ldots,1]^T$ and $\mathrm{Tol}=\sqrt{e}$, where
  $e=2^{-52}$ denotes the machine epsilon.
\item The second example is \texttt{bcspwr10}, also from the
  Harwell--Boeing collection. This is a symmetric Hermitian sparse
  matrix of order $n=5,300$ with $nnz=21,842$. We set $t=2$,
  $b_0=[1,0,\ldots,0,1]^T$ and $\mathrm{Tol}=10^{-5}$.
\item The third example, \texttt{gr\_30\_30}, again of the
  Harwell--Boeing collection, is a symmetric matrix arising when
  discretizing the Laplacian operator using a nine-point stencil on a
  $30 \times 30$ grid. This yields an order $n=900$ sparse matrix with
  $nnz=7,744$ nonzero elements. Here, we choose $t=2$ and compute
  $\e^{-tA} \e^{tA} b_0$, where $b_0=[1,1,\ldots,1]^T$, in two steps:
  first the forward step computes $w=e^{tA}b_0$ and then we use the
  result $w$ as the operand vector for the reverse part $e^{-tA}w$.
  The result should approximate $b_0$ with $\mathrm{Tol}=10^{-14}$. 
\item The final example uses \texttt{helm2d03} from the
  \texttt{GHS\_indef} collection (see \cite{davis08uof} for more
  details), which describes the Helmholtz equation $-\Delta^T \nabla u
  - 10000 u = 1$ on a unit square with Dirichlet $u=0$ boundary
  conditions. The resulting symmetric sparse matrix of order
  $n=392,257$ has $nnz=2,741,935$ nonzero elements. We compute
  $\e^{tA} b_0 + t \varphi_1(tA) b_1$, where $t=2$ and
  $b_0=b_1=[1,1,\ldots,1]^T$.  In this test, only the codes \phiv,
  \phip\ and \phipm\ are compared, with $\mathrm{Tol}=\sqrt{e}$.  
\end{itemize}
Only the third example has a known exact solution. To compute the
exact solutions for the other three examples we use \phiv, \phipm\ and
\expmv\, from \cite{almohy10cta}, with a small tolerance so that all
methods agree to a suitable level of accuracy. We report relative
errors computed at the using
\begin{equation*}
 \mathrm{error} =
 \left\|\frac{u_{\mathrm{exact}}-u_{\mathrm{approx}}}{u_{\mathrm{exact}}}
   \right\|, 
\end{equation*}
where $u_{\mathrm{exact}}$ and $u_{\mathrm{approx}}$ are the exact and
approximate solutions. When the exact solution has components which
are zero they are removed from the relative error calculations. In the
first three comparisons we measure the average speedup and error of
each of the codes relative to \expv; in the final comparison we
measure relative to \phiv. The \texttt{tic} and \texttt{toc} functions
from \matlab\ are used to compute the timings. We ran the comparisons
100 times to compute the average speedup.  We summarize our findings
in Table~\ref{tab:exp_exper}.

\begin{table}[ht!]
  \centering
  \caption{Comparisons of the average speedup of \phiv, \phip\ and
    \phipm\ relative to \expv\ and the relative errors on four 
    matrices taken from the University of Florida sparse matrix
    collection. \label{tab:exp_exper}}
  \begin{tabular}{|c|cc|cc|cc|cc|}
    \hline
    & \multicolumn{2}{c|}{\texttt{orani678}} 
    & \multicolumn{2}{c|}{\texttt{bcspwr10}} 
    & \multicolumn{2}{c|}{\texttt{gr\_30\_30}} 
    & \multicolumn{2}{c|}{\texttt{helm2d03}} \\ \hline
   code & speed & error &  speed & error & speed & error &
    speed & error \\ \hline 
    \expv\ & 1 & $3.1\times 10^{-9}$ & 1 & $5.8\times 10^{-14}$ 
    & 1 & $1.2\times 10^{-7}$ &  & \\
    \phiv\ & 0.96 & $1.6 \times 10^{-7}$ & 0.97 & $1.0\times 10^{-14}$
    & 0.94 & $2.1\times 10^{-7}$ & 1 & $4.3 \times 10^{-8}$ \\
    \phip\ & 0.95 & $3.5 \times 10^{-11}$ & 3.94 & $8.0\times 10^{-13}$
    & 2.69 & $1.8\times 10^{-6}$ & 2.59 & $1.9 \times 10^{-7}$\\
    \phipm\ & 1.35 & $2.4 \times 10^{-11}$ & 6.10 & $5.7\times 10^{-5}$ 
    & 3.59 & $3.9\times 10^{-6}$ & 4.63 & $1.9 \times 10^{-7}$ \\\hline
 \end{tabular}
\end{table}

\subsubsection*{Experiment 2}

In these computations we evaluate $\varphi_0(tA)b_0 + t
\varphi_1(tA)b_1 + \cdots + t^4\varphi_4(tA)b_4$, where
$b_0=\cdots=b_4=[1,1,\ldots,1]^T$, with the codes \phip\ and \phipm.
This comparison gauges the efficiency gains achieved by allowing the
Krylov subspace size to vary. The implementations are identical except
for the fact that \phipm\ can vary $m$ as well. We use the four sparse
matrices described above, with the same values of $t$ and Tol, except
for the sparse matrix \texttt{gr\_30\_30} with
$\mathrm{Tol}=\sqrt{e}$, is used and we only compute the forward part
of the problem. We summarize our findings in
Table~\ref{tab:phi_exper}.

\begin{table}[ht!]
  \centering
  \caption{Comparisons of the average speedup of \phipm\ relative to
    \phip\ and the relative errors on four large sparse matrices
    taken from the University of Florida sparse matrix collection. 
    \label{tab:phi_exper}}
  \begin{tabular}{|c|cc|cc|cc|cc|}
    \hline
    & \multicolumn{2}{c|}{\texttt{orani678}} 
    & \multicolumn{2}{c|}{\texttt{bcspwr10}} 
    & \multicolumn{2}{c|}{\texttt{gr\_30\_30}} 
    & \multicolumn{2}{c|}{\texttt{helm2d03}} \\ \hline
    code & speed & error &  speed & error &  speed & error &
    speed & error \\ \hline 
    \phip\ & 1 & $8.7\times 10^{-13}$ & 1 & $2.5\times 10^{-10}$  
    & 1 & $4.6\times 10^{-13}$ & 1 & $5.2\times 10^{-8}$ \\
    \phipm\ & 1.37 & $2.1\times 10^{-12}$ & 1.35 & $4.2\times 10^{-5}$ 
    & 1.16 & $6.0\times 10^{-13}$ & 1.87 & $5.2\times 10^{-8}$ \\ \hline
 \end{tabular}
\end{table}

\subsubsection*{Discussion of the results}

These comparisons show that in all cases the \phipm\ code is more
efficient, in some cases by a considerable margin.  Summarizing,
adapting both the dimension of Krylov subspace as well as the length
of the time steps significantly increases overall efficiency. Given
that in an implementation of an exponential integrator \phipm\ would
be called several times in a step over many steps during the
integration, this increase in efficiency can often lead to very large
overall computational gains.


\section{Conclusion and future work}
\label{sect:concl}

The \phipm\ function is an efficient routine which computes the action
of linear combinations of $\varphi$-functions on operand vectors. The
implementation combines time stepping with a procedure to adapt the
Krylov subspace size. It can be considered as an extension of the
codes provided in \expokit\ and \mathematica.

The $\varphi$-functions are the building blocks of exponential
integrators. An implementation of the algorithm in a lower-level
language will be useful in this context; this is work in progress. We
are also working on the implementation of exponential integrators
which use the \phipm\ routine described in this paper and hope to
report on this shortly.

We intend to improve the code over time. Some issues which we plan to
investigate have already been mentioned. One of them is the issue of
stability, especially in view of the error estimates in
Figure~\ref{fig:ex1error} which might point to stability problems. Our
choice to compute the $\varphi$-function of the reduced matrix~$H_m$
by adding some rows and columns and then computing the matrix
exponential may exacerbate any instabilities. Perhaps it is better to
compute the $\varphi$-function of~$H_m$ directly. This also allows us
to exploit the fact that~$H_m$ is symmetric if the matrix~$L$ in the
original differential equation is symmetric, for instance by using
rational Chebyshev approximants. In any case, regardless of whether we
augment the matrix~$H_m$ or not, we do not need to compute the matrix
function to full precision. It should also be possible to exploit the
fact that $H_m$ is Hessenberg.

We also intend to modify \phipm\ to take advantage of recent advances
in parallel processing technology, specifically the use of graphics
cards accessed using programming languages such as CUDA, which provide
promise of significant computational improvements. We are currently
investigating the application of Krylov-based methods to option
pricing problems, where the governing PDEs are often linear. The
corresponding discretized ODEs often take the form of
Equation~\eqref{eq:nonautonomous}, which can naturally be computed
using \phipm. Finally, as mentioned in the introduction, there are
alternatives to the (polynomial) Krylov method considered in this
paper. We plan to study other methods and compare them against the
method introduced here.  Competitive methods can be added to the code,
because we expect that the performance of the various methods depends
strongly on the characteristics of both the problem and the
exponential integrator.


\subsection*{Acknowledgements}
  The authors thank Nick Higham, Karel in 't Hout, Brynjulf Owren,
  Roger Sidje and the anonymous referee for helpful discussions.

\bibliographystyle{chicago}
\bibliography{../bibl/geom_int,../bibl/hamilt,../bibl/expintbib,../bibl/brynbib,../bibl/glm,../bibl/jitse}

\end{document}